\documentclass[12pt]{article}
\usepackage{amsmath,amssymb}
\usepackage[T1]{fontenc}
\usepackage{latexsym}
\usepackage{epsfig}
\usepackage{psfrag}
\usepackage[utf8]{inputenc}
\usepackage[a4paper, top=3cm, bottom=3cm ]{geometry}

\newtheorem{theorem}{Theorem}[section]

\newtheorem{corollary}{Corollary}[section]
\newtheorem{prop}{Proposition}[section]

\title
{THE BERGMAN KERNEL: EXPLICIT FORMULAS, DEFLATION, \\ LU QI-KENG PROBLEM
AND JACOBI POLYNOMIALS}

\author{\normalsize Tomasz Beberok \\
\small Faculty of Mathematics and Computer Science, Jagiellonian University,\\
\small Lojasiewicza 6, 30-048 Krakow, Poland \\}

\date{}

\begin{document}

\begin{center}
  \textbf{THE BERGMAN KERNEL: EXPLICIT FORMULAS, DEFLATION, \\ LU QI-KENG PROBLEM
AND JACOBI POLYNOMIALS}
\end{center}
\vskip1em
\begin{center}
  Tomasz Beberok
\end{center}

\vskip2em

In this paper we investigate the Bergman kernel function for intersection of two complex ellipsoids $$\{(z,w_1,w_2) \in \mathbb{C}^{n+2} : |z_1|^2 + \cdots + |z_n|^2 + |w_1|^q < 1, \quad |z_1|^2 + \cdots + |z_n|^2  + |w_2|^r < 1\}.$$
\vskip1em

\textbf{Keyword:} Lu Qi-Keng problem, Bergman kernel, Routh-Hurwitz theorem, Jacobi polynomials
\vskip1em
\textbf{AMS Subject Classifications:} 32A25;  33D70

\section{Introduction}
Let $D$ be a bounded domain in $\mathbb{C}^n$. The Bergman space $L^2_a(D)$ is the space of all square integrable holomorphic functions on $D$. Then the Bergman kernel $K_D(z,w)$ is defined \cite{BE} by
\begin{align*}
K_D(z,w)= \sum_{j=0}^{\infty} \Phi_j(z) \overline{\Phi_j(w)}, \quad (z,w) \in D \times D,
\end{align*}
where $\{\Phi_j(\cdot) \colon  j = 0, 1, 2, . . .\}$ is a complete orthonormal basis for $L^2_a(D)$. It is defined for arbitrary domains, but it is hard to obtain concrete representations for the Bergman kernel except for special cases like a Hermitian ball or polydisk. Refer to \cite{KRA} for more on this topic. \newline \indent In 2015 \cite{TB} the author of this paper computed the Bergman kernel for $D_1^{q,r}=\{z=(z_1,w_1,w_2) \in \mathbb{C}^3 \colon |z_1|^2 + |w_1|^q < 1, \quad |z_1|^2 + |w_2|^r < 1\}$ explicitly. The goal of this paper is to extend the result in \cite{TB} to higher dimensional case, namely to $D_n^{q,r}=\{(z,w_1,w_2) \in \mathbb{C}^{n+2} : |z_1|^2 + \cdots + |z_n|^2 + |w_1|^q < 1, \quad |z_1|^2 + \cdots + |z_n|^2  + |w_2|^r < 1\}.$ \newline \indent This paper will be organized as follows. In Section 2, we will compute explicit formula of the Bergman kernel for $D_n^{q,r}$ and we show deflation identity between domains $D_n^{q,r}$ and $D_{1/n}^{q,r}=\{z=(z_1,w_1,w_2) \in \mathbb{C}^3 \colon |z_1|^{2/n} + |w_1|^q < 1, \quad |z_1|^{2/n} + |w_2|^r < 1\}$. In Section 3, we show some relation between Bergman kernel for $D_n^{2,2}$ and Jacobi polynomials. In Section 4, we investigate the Lu Qi-Keng problem for $D_n^{2,2}$. In the final section, we consider the Bergman kernel for $\Omega_n^r:=\left\{ (z,w) \in \mathbb{C} \times \mathbb{C}^n \colon |z|^2 + |w_1|^r <1, \ldots,|z|^2 + |w_n|^r <1 \right\}$.

\section{Bergman Kernel}
Let $\zeta=(z,w_1,w_2) \in \mathbb{C}^{n+2}$. Put $\Phi_\kappa(\zeta)= \zeta^{\kappa}=z^{\alpha}w_1^{\gamma_1}w_2^{\gamma_2}$ for each multi-index $\kappa=(\alpha,\gamma_1,\gamma_2)$. Since $D_n^{q,r}$ is complete Reinhardt domain, the collection of $\{\Phi_{\kappa} \}$ such that each $\alpha_i \geq 0$ and $\gamma_j \geq 0$ is a complete orthogonal set for $L^2(D_n^{q,r})$.

\begin{prop}\label{pr1} Let $\alpha_i \in \mathbb{Z}_{+}$ for $i=1,\ldots,n$ and $\gamma_1 \geq 0$, $\gamma_2 \geq 0$. Then, we have
$${\left\| z^{\alpha} w_1^{\gamma_1} w_2^{\gamma_2} \right\|}^2_{L^2(D_n^{q,r})}= \frac{  \pi^{n+2} \Gamma(\frac{2\gamma_1 + 2}{q} +\frac{2\gamma_2 + 2}{r} + 1) \prod_{i=1}^n \Gamma(\alpha_i + 1)   }{  (\gamma_2 + 1) (\gamma_2 +1 ) \Gamma( \frac{2\gamma_1 + 2}{q} +\frac{2\gamma_2 + 2}{r} + |\alpha| + n + 1)},$$
where $|\alpha|=\alpha_1 + \cdots + \alpha_n$.
\end{prop}

\begin{proof}
$$\| z^{\alpha} w_1^{\gamma_2} w_2^{\gamma_3}  \|^2_{L^2(D_n^{q,r})} = \int\limits_{D_n^{q,r}} |z|^{2\alpha} |w_1|^{2\gamma_2} |w_2|^{2\gamma_3}\, dV(z)dV(w)$$
We introduce polar coordinate in each variable by putting $z=r e^{i\theta}$, $w_1=s_1 e^{i\lambda_1}$, $w_2=s_2 e^{i\lambda_2}$. After doing so, and integrating out the angular variables we have

$$(2 \pi)^{n+2} \int_{Re(D_n^{q,r})} r^{2\alpha + 1} s_1^{2\gamma_1 + 1} s_2^{2\gamma_2 + 1}\, dV(r)dV(s),$$
where $Re(D_n^{q,r})=\left\{(r,s) \in  \mathbb{R}_{+}^{n} \times \mathbb{R}_{+}^{2} : \|r\|^2+s_1^q<1, \quad \|r\|^2+s_2^r<1 \right\}$. Next we use spherical coordinates in the $r$ variable to obtain

$$(2 \pi)^{n+2} \int_{S\ast Re(D_n^{q,r})} \int_{\mathbf{S}_{+}^{n_1-1}}   \rho^{2|\alpha| + 2n - 1} \omega^{2\alpha+1} s_1^{2\gamma_1 + 1} s_2^{2\gamma_2 + 1}\, d\sigma(\omega)dV(s)d\rho,$$
where $S\ast Re(D_n^{q,r})=\left\{(\rho,s) \in  \mathbb{R}_{+} \times \mathbb{R}_{+}^{2} : \rho^2+s_1^q<1, \quad \rho^2+s_2^r<1 \right\}$.  Using (6) form Lemma 1 in \cite{DA2} we obtain

$$(2 \pi)^{n+2} \frac{\beta(\alpha+1)}{2^{n-1}} \int_0^1 \int_0^{(1-\rho^2)^{1/q}} \int_0^{(1-\rho^2)^{1/r}} \rho^{2|\alpha| + 2n - 1} s_1^{2\gamma_1 + 1} s_2^{2\gamma_2 + 1}\, ds_1 ds_2 d\rho,$$
where $\beta(\alpha)=\frac{\prod_{i=1}^n \Gamma(\alpha_i)}{\Gamma(\alpha_1 + \cdots + \alpha_n)}$. Integrating out of $s_1$ and $s_2$ variables, we have

$$ \frac{(2 \pi)^{n+2}\beta(\alpha+1)}{2^{n+1}(\gamma_1+1)(\gamma_2+1)}   \int_0^1 \rho^{2|\alpha| + 2n - 1}  (1-\rho^2)^{\frac{2\gamma_1 + 2}{q} +\frac{2\gamma_2 + 2}{r}} \, d\rho$$
After a little calculation using well known fact $$\int_0^1 x^a(1-x^p)^b \,dx = \frac{\Gamma((a+1)/p) \Gamma(b+1)}{p  \Gamma((a+1)/p + b + 1) } ,$$  we obtain desired result.
\end{proof}

Now we discuss the Bergman kernel for $D_n^{q,r}$.

\begin{theorem}
  The Bergman kernel for $D_n^{q,r}$ is given by
\begin{align*}
K_{D_n^{q,r}} ((z, w_1, w_2),(\eta, \xi_1, \xi_2))=\frac{L_n^{q,r}(a,b)}{\pi^{n+2}(1- z_1\overline{\eta}_1 -  \cdots - z_n\overline{\eta}_n)^{\frac{2}{q} +\frac{2}{r} + 1} },
\end{align*}
where
\begin{align*}
  a=\frac{w_1\overline{\xi}_1}{(1- z_1\overline{\eta}_1 -  \cdots - z_n\overline{\eta}_n)^{\frac{2}{q}}}, \quad b=\frac{ w_2\overline{\xi}_2}{(1- z_1\overline{\eta}_1 -  \cdots - z_n\overline{\eta}_n)^{\frac{2}{r}}} \\  L_{n+1}^{q,r}(x,y)=nL_n^{q,r}(x,y)+ \frac{2}{q}\frac{\partial}{\partial x}x L_n^{q,r}(x,y) + \frac{2}{r}\frac{\partial}{\partial y}y L_n^{q,r}(x,y) \text{ and} \\  L_1^{q,r}(x,y)=\frac{2 (q (1-x) (y+1)+r (x+1) (1-y))}{q r (1-x)^3 (1-y)^3}
\end{align*}
\end{theorem}

\begin{proof}
  By Proposition \ref{pr1}, we have
\begin{align*}
&K_{D_n^{q,r}} ((z, w_1, w_2),(\eta, \xi_1, \xi_2)) \\&= \frac{1}{\pi^{n+2}} \sum_{\alpha =0 }^{ \infty} \sum_{\gamma_1,\gamma_2 =0 }^{ \infty} \frac{ (\gamma_1 + 1) (\gamma_2 +1 ) \Gamma( \frac{2\gamma_1 + 2}{q} +\frac{2\gamma_2 + 2}{r} + |\alpha| + n + 1)}{\Gamma(\frac{2\gamma_1 + 2}{q} +\frac{2\gamma_2 + 2}{r} + 1) \prod_{i=1}^n \Gamma(\alpha_i + 1) } \mu^{\alpha} \nu_1^{\gamma_1} \nu_2^{\gamma_2} ,
\end{align*}
where $\mu^{\alpha}=(z_1\overline{\eta}_1)^{\alpha_1}\cdot \cdots \cdot (z_n\overline{\eta}_n)^{\alpha_n}$ and $\nu_1=w_1\overline{\xi}_1,\, \nu_2=w_2\overline{\xi}_2$. Sum out of each $\alpha_i$, we have

\begin{align*}
 \frac{1}{\pi^{n+2}(1- \tau)^{\frac{2}{q} +\frac{2}{r} + 1} }  \sum_{\gamma_1,\gamma_2 =0 }^{ \infty} \frac{ (\gamma_1 + 1) (\gamma_2 +1 ) \Gamma(\frac{2\gamma_1 + 2}{q} +\frac{2\gamma_2 + 2}{r} + n + 1)}{\Gamma(\frac{2\gamma_1 + 2}{q} +\frac{2\gamma_2 + 2}{r} + 1) (1-\tau)^{\frac{2\gamma_1}{q} +\frac{2\gamma_2}{r}} }  \nu_1^{\gamma_1} \nu_2^{\gamma_2} ,
\end{align*}
where $\tau=z_1\overline{\eta}_1 +  \cdots + z_n\overline{\eta}_n$. Now we will consider sequence of functions $L_n^{q,r}$ defined as follows
\begin{align*}
   L_n^{q,r}(x,y)=\sum_{\gamma_1,\gamma_2 =0 }^{ \infty} \frac{ (\gamma_1 + 1) (\gamma_2 +1 ) \Gamma(\frac{2\gamma_1 + 2}{q} +\frac{2\gamma_2 + 2}{r} + n + 1)}{\Gamma(\frac{2\gamma_1 + 2}{q} +\frac{2\gamma_2 + 2}{r} + 1) }  x^{\gamma_1} y^{\gamma_2}
\end{align*}
Using the identity $\Gamma(t+1)=t\Gamma(t)$ we easily obtain recursion formula
\begin{align*}
   L_{n+1}^{q,r}(x,y)=nL_n^{q,r}(x,y)+ \frac{2}{q}\frac{\partial}{\partial x}x L_n^{q,r}(x,y) + \frac{2}{r}\frac{\partial}{\partial y}y L_n^{q,r}(x,y)
\end{align*}
Moreover
\begin{align*}
   L_1^{q,r}(x,y)=\sum_{\gamma_1,\gamma_2 =0 }^{ \infty}  (\gamma_1 + 1) (\gamma_2 +1 ) \left(\frac{2\gamma_1 + 2}{q} +\frac{2\gamma_2 + 2}{r} + 1\right)  x^{\gamma_1} y^{\gamma_2}
\end{align*}
Hence
\begin{align*}
   L_1^{q,r}(x,y)=\frac{2 (q (1-x) (y+1)+r (x+1) (1-y))}{q r (1-x)^3 (1-y)^3}
\end{align*}
which completes the proof.
\end{proof}
\subsection{Deflation}
Now we will consider following domains
\begin{align*}
  D_{1/n}^{q,r}=\{(z,w_1,w_2) \in \mathbb{C}^3 \colon |z|^{2/n} + |w_1|^q < 1, \quad |z|^{2/n} + |w_2|^r < 1\}
\end{align*}
Similarly as in Sect. 2, we have
\begin{prop}\label{pr2}
Let $\alpha \geq 0 $, $\gamma_1 \geq 0$ and $\gamma_2 \geq 0$. Then, we have
$${\left\| z^{\alpha} w_1^{\gamma_1} w_2^{\gamma_2} \right\|}^2_{L^2(D_{1/n}^{q,r})}= \frac{ n \pi^{3} \Gamma(\frac{2\gamma_1 + 2}{q} +\frac{2\gamma_2 + 2}{r} + 1) \Gamma(n\alpha + n)   }{  (\gamma_2 + 1) (\gamma_2 +1 ) \Gamma( \frac{2\gamma_1 + 2}{q} +\frac{2\gamma_2 + 2}{r} + n\alpha + n + 1)}.$$
\end{prop}
\noindent Therefore
\begin{align*}
&K_{D_{1/n}^{q,r}} ((0, w_1, w_2),(0, \xi_1, \xi_2)) \\&=  \sum_{\gamma_1,\gamma_2 =0 }^{ \infty} \frac{ (\gamma_1 + 1) (\gamma_2 +1 ) \Gamma( \frac{2\gamma_1 + 2}{q} +\frac{2\gamma_2 + 2}{r} + n + 1)}{\pi^{3}n!\Gamma(\frac{2\gamma_1 + 2}{q} +\frac{2\gamma_2 + 2}{r} + 1) }  (w_1\overline{\xi}_1)^{\gamma_1} (w_2\overline{\xi}_2)^{\gamma_2}
\end{align*}
Hance
\begin{align*}
K_{D_{1/n}^{q,r}} ((0, w_1, w_2),(0, \xi_1, \xi_2)) = \frac{1}{\pi^{3}n!} L_n^{q,r}(w_1\overline{\xi}_1,w_2\overline{\xi}_2)
\end{align*}
Comparing formulas for $K_{D_{1/n}^{q,r}}$ and $K_{D_{n}^{q,r}}$, we obtain following deflation identity
\begin{prop}\label{deflation} For every $n \in \mathbb{N}$ and every positive numbers $q$ and $r$, we have
 \begin{align*}
 \frac{n!}{\pi^{n-1}} K_{D_{1/n}^{q,r}} ((0, w_1, w_2),(0, \xi_1, \xi_2)) = K_{D_{n}^{q,r}} ((0,\ldots,0, w_1, w_2),(0,\ldots,0, \xi_1, \xi_2))
\end{align*}
\end{prop}
 \noindent Note that, we have some kind of deflation result similar to that obtained in \cite{BS}.
\section{Some representations of Bergman kernel for $D_{n}^{2,2}$ }
Jacobi polynomials are a class of classical orthogonal polynomials. They are orthogonal with respect to the weight $(1 - x)^k(1 + x)^l$ on the interval $[-1, 1]$. For $k ,\, l >-1$ the Jacobi polynomials are given by the formula
\begin{align*}
  P_d^{(k,l)}(z)= \frac{(-1)^d}{2^d d!}(1-z)^{-k}(1+z)^{-l} \frac{\partial}{\partial z^d}\left\{ (1-z)^{k}(1+z)^{l}(1-z^2)^d  \right\}.
\end{align*}
The Jacobi polynomials are defined via the hypergeometric function as follows
\begin{equation}\label{jacobi}
  P_d^{(k,l)}(z)= \frac{(k+1)_d}{d!} {}_2F_1\left(-d,1+k+l+d;k+1,\frac{1-z}{2}\right),
\end{equation}
where $(k+1)_d$ is Pochhammer's symbol and ${}_2F_1$ is Gaussian or ordinary hypergeometric function defined for $|z| < 1$ by the power series
\begin{align*}
  {}_2F_1\left(a,b;c;z\right)= \sum_{n=0}^{\infty} \frac{(a)_n (b)_n}{n!(c)_n} z^n
\end{align*}
Appell series $F_1$ defined by
\begin{align*}
  F_1(a;b,b';c,x,y)=\sum_{n=0}^{\infty} \sum_{m=0}^{\infty} \frac{(a)_{m+n} (b)_n (b')_m}{n!m!(c)_{n+m}} x^n y^m, \quad (|x|<1,|y|<1)
\end{align*}
is one of a natural two-variable extension of hypergeometric series ${}_2F_1$. \newline The following reduction formulas can be proved (see, for details \cite{HTF}, p. 238-239)
\begin{eqnarray}
  F_1(a;b,b';b+b';x,y)&=&(1-y)^{-a} {}_2F_1(a,b;b+b';(x-y)/(1-y)) \label{reduction1} \\
  F_1(a;b,b';c;x,x)&=&{}_2F_1(a,b+b';c;x)\label{reduction2}
\end{eqnarray}
Comparing functions $F_1$ and $L_n^{2,2}$ it is easy to see
\begin{align*}
  \frac{\Gamma(3+n)}{2}F_1(3+n;2,2;3;x,y)=L_n^{2,2}(x,y)
\end{align*}
The following is the main theorem of this section.
\begin{theorem}\label{Rep}
  The Bergman kernel for $D_{n}^{2,2}$ can be expressed in the following ways
\begin{description}
\item[(i)]  $K_{D_n^{2,2}} ((0, w_1, w_2),(0, \xi_1, \xi_2))=$ \begin{align*}\frac{\Gamma(3+n)}{2\pi^{n+2}} \sum\limits_{i=0}^n \sum\limits_{k=0}^{n-i} \binom{n}{i} \binom{n-1}{k} \frac{(2)_i(2)_k (w_1\overline{\xi}_1)^i (w_2\overline{\xi}_2)^k}{(3)_{i+k} (1-w_1\overline{\xi}_1)^{2+i} (1-w_2\overline{\xi}_2)^{2+k}}\end{align*}
\item[(ii)] For every $m \in  \mathbb{N} \cup \{0\}$, we have
        \begin{align*}
        K_{D_{2m+1}^{2,2}} ((0, w_1, w_2),(0, \xi_1, \xi_2))=&\frac{C_m (2+m)(2-\nu_1 - \nu_2)P_m^{(\frac{3}{2}, \frac{1}{2})} (\frac{x'+1}{1-x'}) }{ \pi^{2m+3} (1-\nu_1 - \nu_2 -\nu_1 \nu_2 )^{m+3} } \\&- \frac{C_m(2m+1) P_m^{(\frac{3}{2}, -\frac{1}{2})} (\frac{x'+1}{1-x'}) }{ \pi^{2m+3}(1-\nu_1 - \nu_2 -\nu_1 \nu_2 )^{m+2} } , \end{align*}
        where $\nu_1=w_1\overline{\xi}_1,\, \nu_2=w_2\overline{\xi}_2$, $x'=\left(\frac{\nu_1 - \nu_2}{2 - \nu_1 - \nu_2}\right)^2$ and $C_m=\frac{\Gamma(4+2m) m!}{6(\frac{5}{2})_m}$ .
\item[(iii)] For every $m \in  \mathbb{N}$, we have $K_{D_{2m}^{2,2}} ((0, w_1, w_2),(0, \xi_1, \xi_2))=$
        \begin{align*}\Gamma(3+2m)m!\frac{ 2P_m^{(\frac{3}{2}, -\frac{1}{2})} (\frac{x'+1}{1-x'}) - (2-\nu_1 - \nu_2) P_{m-1}^{(\frac{3}{2},\frac{1}{2})} (\frac{x'+1}{1-x'}) }{  \pi^{2m+2}6(\frac{5}{2})_{m-1} (1-\nu_1 - \nu_2 -\nu_1 \nu_2 )^{m+2} }, \end{align*}

\item[(iv)]  If $w_1\overline{\xi}_1 = w_2\overline{\xi}_2 $ then \begin{align*}K_{D_n^{2,2}} ((0, w_1, w_2),(0, \xi_1, \xi_2))=\frac{\Gamma(3+n)(3+n w_1\overline{\xi}_1)}{6 \pi^{n+2} (1-w_1\overline{\xi}_1)^{4+n}} \end{align*}
\end{description}
\end{theorem}
\begin{proof}
  For the proof of (i), if we apply the recursion formula (see \cite{XW})
  \begin{align*}
    F_1(a+n;b,b';c;x,y)=&\sum\limits_{i=0}^n \sum\limits_{k=0}^{n-i} \binom{n}{i} \binom{n-1}{k} \frac{(b)_i(b')_k}{(c)_{i+k}} \\&\times x^i y^k F_1(a+i+k;b+i,b'+k;c+i+k;x,y),
  \end{align*}
then using formula
\begin{align*}
  F_1(3+i+k;2+i,2+k;3+i+k;x,y)=\frac{1}{(1-x)^{2+i}(1-y)^{2+k}}
\end{align*}
we obtain (i). \\ In order to prove the second statements we need the following well-known contiguous relation
\begin{align*}
  cF_1(a;b,b';c;x,y)-(c-a)F_1(a;b,b';c+1;x,y)& \\-aF_1(a+1;b,b';c+1;x,y)&=0
\end{align*}
It follows that
\begin{align*}
  F_1(4+2m;2,2;3;x,y)=&-\frac{2m+1}{3}F_1\left(2m+4;2,2;4;x,y\right)\\&+\frac{4+2m}{3}
  F_1\left(2m+5;2,2;4;x,y\right)
\end{align*}
By (\ref{reduction1}), we have
\begin{align*}
  F_1(4+2m;2,2;3;x,y)=&-\frac{2m+1}{3(1-y)^{4+2m}}{}_2F_1\left(2m+4,2;4;\frac{x-y}{1-y}\right)\\&+\frac{4+2m}{3(1-y)^{5+2m}}
  {}_2F_1\left(2m+5,2;4;\frac{x-y}{1-y}\right)
\end{align*}
Then by (see, for details \cite{HTF}, p. 66)
\begin{align*}
  {}_2F_1(a,b;2b;z)=\left(1-\frac{z}{2}\right)^{-a} {}_2F_1\left(\frac{a}{2},\frac{a+1}{2};b+\frac{1}{2};[z/(2-z)]^2\right),
\end{align*}
we have
\begin{align*}
  F_1(4+2m;2,2;3;x,y)=&-\frac{(2m+1)2^{4+2m}}{3(2-x-y)^{4+2m}}{}_2F_1\left(m+2,m+\frac{5}{2};\frac{5}{2};(x')^2\right)
  \\&+\frac{(4+2m)2^{5+2m}}{3(2-x-y)^{5+2m}}
  {}_2F_1\left(m+\frac{5}{2},m+3;\frac{5}{2};(x')^2\right),
\end{align*}
where $x'=\left(\frac{x-y}{2-x-y}\right)^2$. Next by (see, for details \cite{HTF}, p. 64)
\begin{align*}
  {}_2F_1(a,b;c;z)=&\left(1-z\right)^{-a} {}_2F_1\left(a,c-b;c;z/(z-1)\right)\\=&\left(1-z\right)^{-b} {}_2F_1\left(c-a,b;c;z/(z-1)\right),
\end{align*}
we can easily get $F_1(4+2m;2,2;3;x,y)=$
\begin{align*}
  &-\frac{2m+1}{3(1-x-y-xy)^{m+2}}{}_2F_1\left(m+2,-m;\frac{5}{2};\left(\frac{x'}{x'-1}\right)^2\right)
  \\&+\frac{(2+m)(2-x-y)}{3(1-x-y-xy)^{m+3}}
  {}_2F_1\left(m+\frac{5}{2},m+3;\frac{5}{2};\left(\frac{x'}{x'-1}\right)^2\right),
\end{align*}
Finally, by (\ref{jacobi}) we obtain (ii). The formula (iii) can be obtained by the same method and we omit the details. \\Now we will prove (iv). This is a formal exercise but we include it for completeness: From (\ref{reduction2})
\begin{align*}
  F_1(3+n;2,2;3;x,y)= {}_2F_1(3+n,4;3;x)
\end{align*}
Now it is relatively easy to compute the following result
\begin{align*}
  {}_2F_1(3+n,4;3;x)=\frac{3+nx}{3(1-x)^{4+n}}
\end{align*}
Thus we prove (iv).
\end{proof}

\section{Lu Qi-Keng problem}
The explicit formula of the Bergman kernel function for the domains $D_n^{q,r}$ enables us to investigate whether the Bergman kernel has zeros in $D_n^{q,r} \times D_n^{q,r}$ or not. We will call this kind of problem a Lu Qi-Keng problem. If the Bergman kernel for a bounded domain does not have zeros, then the domain will be called a Lu Qi-Keng domain. \newline \indent By Theorem 1.2 form \cite{TB} if $n=1$, then $D_n^{q,r}$ is a Lu Qi-Keng domain for all positive real numbers $q$ and $r$. Combining deflation identity (Proposition \ref{deflation}) and Proposition 4.4 from \cite{TB}, if $n=2$ and $q=r$, then $D_2^{r,r}$ is not a  Lu Qi-Keng domain. Using the same method as in \cite{TB} we will prove that it is also true for $n=3$. \newline \indent Denote by
\begin{align*}
  G(x,y)=&3 r^3 (1-x)^3 (1-y)^3+22r^2 (1-x)^2 (1-y)^2 (1-x y) \\&+24 r (1-x) (1-y) \left(x (2 x+1) y^2+(x-8) x y+x+y+2\right) \\&+ 8 (1-x y) \left(x^2 y (4 y+7)+x^2 +y^2 + xy (7 y-38)+7x +7y +4\right).
\end{align*}
Since $$G(x,y)= \frac{r^3 (1-x)^5 (1-y)^5}{2} L_3^{r,r}(x,y)$$ then the Bergman kernel $K_{D_3^{r,r}}$ has zero inside $D_2^{r,r} \times D_2^{r,r}$ if polynomial $G(\epsilon x, \epsilon y)$ does not satisfy the stability property (see \cite{TB}, for details) for some $0 < \epsilon < 1$. By following  \v{S}iljak and  Stipanovi\'{c} \cite{Siljak} we consider polynomial
\begin{align*}
   z^3 G(e^{i\eta},1/z)=d(z)=&d_3 z^3 + d_2z^2 + d_1 z + d_0,
\end{align*}
where
\begin{align*}
  &d_0 = 3 r^3 (-1 + t)^3 - 22 r^2 (-1 + t)^2 t + 24 r t (-1 - t + 2 t^2) -
         8 t (1 + 7 t + 4 t^2), \\
  &d_1 = 8 - 9 r^3 (-1 + t)^3 + 336 t^2 - 56 t^3 +
         22 r^2 (-1 + t)^2 (1 + 2 t) \\  &\quad \quad - 24 r (1 - 10 t + 8 t^2 + t^3), \\
  &d_2 = 9 r^3 (-1 + t)^3 - 22 r^2 (2 - 3 t + t^3) - 8 (-7 + 42 t + t^3)
          \\  &\quad \quad        -24 r (1 + 8 t - 10 t^2 + t^3), \\
  &d_3 = 22 r^2 (-1 + t)^2 - 3 r^3 (-1 + t)^3 - 24 r (-2 + t + t^2) +
          8 (4 + 7 t + t^2)\\ &\text{and} \quad t=e^{i\eta}.
\end{align*}
With the polynomial $d(z)$ we associate the Schur-Cohn $3\times3$ matrix
$$M(e^{i\eta})=\left(
   \begin{array}{ccc}
     d_{11} & d_{12} & d_{13} \\
     d_{21} & d_{22} & d_{23} \\
     d_{31} & d_{32} & d_{33} \\
   \end{array}
 \right)
,$$
$d_{jk}=\sum_{l=1}^j(d_{m-j+l} \overline{d}_{m-k+l} - \overline{d}_{j-l} d_{k-l}),$ where $1 \leq j \leq k$. The matrix $M(e^{i\eta})$ is defined when $j > k$ to become Hermitian. After some calculation  (with the help of a computer program Maple or Mathematica), we have

    \begin{align*}
      \det M(e^{i\eta})=& -130459631616  r^3 \sin^{12}\left(\eta/2\right) \left[  \sum_{n=0}^3 g_n(r) \cos(n \eta) \right],
    \end{align*}
where
       \begin{align*}
          & g_0(r)=26624 - 24672 r^2 + 15724 r^4 - 2430 r^6 \\
          & g_1(r)=12288 + 22496 r^2 - 20822 r^4 + 3645 r^6 \\
          & g_2(r)=-512 + 2208 r^2 + 5012 r^4 - 1458 r^6 \\
          & g_3(r)=r^2 (-32 + 86 r^2 + 243 r^4)
       \end{align*}
 Since  \begin{align*}\sum_{n=0}^3 g_n(r)=38400 \end{align*}
 Therefore it is easy to see that for every $r > 0$ there exist $\eta$ such that
$\det M(e^{i\eta})<0$. Hence there exist $1 > \epsilon > 0$, such that polynomial $G(\epsilon x, \epsilon y)$ does not satisfy the stability property (see \cite{Siljak}). As a consequence of above consideration, we have following corollary
\begin{corollary}
For any $r > 0$, domains $D_3^{r,r}$ and $D_{1/3}^{r,r}$ are not Lu Qi-Keng.
\end{corollary}
Now we will study Lu Qi-Keng problem for $D_n^{2,2}$ in the case when $n>3$. By representation (iv) from Theorem \ref{Rep}
\begin{align*}K_{D_n^{2,2}} ((0, w_1, w_1),(0, \xi_1, \xi_1))=\frac{\Gamma(3+n)(3+n w_1\overline{\xi}_1)}{6 \pi^{n+2} (1-w_1\overline{\xi}_1)^{4+n}} .\end{align*}
Hence
\begin{align*}K_{D_n^{2,2}} \left((0, i\sqrt{3/n}, i\sqrt{3/n} ),(0, -i\sqrt{3/n}, -i\sqrt{3/n})\right)=0. \end{align*}
For brevity, we shall summarize these last statements by saying that
\begin{prop}
  Domain $D_n^{2,2}$ is Lu Qi-Keng if and only if $n=1$.
\end{prop}
At the end of this section we would like to present following relations between zeros of the Bergman kernel for domains $D_n^{q,r}$ and $D_{1/n}^{q,r}$.
  \begin{prop}For any positive real numbers $q$ and $r$
     \begin{description}
           \item[(i)] If $K_{D_{n}^{q,r}}$ has zeros, then $K_{D_{1/n}^{q,r}}$ also has zeros,
           \item[(ii)] If $D_{1/n}^{q,r}$ is Lu Qi-Keng domain, then $D_n^{q,r}$ is Lu Qi-Keng domain.
     \end{description}
  \end{prop}
\begin{proof}
 Note that the zero set is a bi-holomorphic invariant object. Since any point $(z,w_1,w_2)\in D_n^{q,r}$ can be mapped equivalently onto the form $(0,\widetilde{w_1},\widetilde{w_2})$ by following automorphism of the $D_n^{q,r}$  $$ D_n^{q,r} \ni(z,w_1,w_2)\mapsto \left(\Psi_a(z), \frac{(1-\|a\|^2)^{1/q}}{(1-\langle z, a\rangle)^{2/q}}w_1, \frac{(1-\|a\|^2)^{1/r}}{(1-\langle z, a\rangle)^{2/r}}w_2\right) \in \mathbb{C}^{n+2} ,$$ where $$\Psi_a(z)=\frac{\left( \frac{\langle z,a\rangle}{1+\sqrt{1-\|a\|^2}}-1 \right)a + z \sqrt{1-\|a\|^2}}{1-\langle z,a\rangle}.$$
  Therefore, we need only consider the zeroes restricted to $\{0\} \times \mathbb{D} \times \mathbb{D}$, where $\mathbb{D}:=\{z \in \mathbb{C}:  |z|<1 \}$. The results now follows from Proposition \ref{deflation}.
\end{proof}

\section{Additional Results}
Our purpose in this section is to consider domains $\Omega_n^r$ defined for every positive real number $r$ by
$$\Omega_n^r:=\left\{ (z,w) \in \mathbb{C} \times \mathbb{C}^n \colon |z|^2 + |w_1|^r <1, \ldots,|z|^2 + |w_n|^r <1 \right\}$$

The reader can see that following proposition is completely analogous to the results presented earlier.
\begin{prop}\label{pr3}
For $j = 1,2,\ldots,n$ let  $\beta_j \geq 0$. Then for $\alpha \geq 0 $, we have
$${\left\| z^{\alpha} w^{\beta} \right\|}^2_{L^2(\Omega_n)}= \frac{ \pi^{n+1} \Gamma\left(\frac{2}{r}(\sum_{j=1}^n \beta_j + n) + 1\right) \Gamma(\alpha + 1)   }{   \Gamma\left(\frac{2}{r}(\sum_{j=1}^n \beta_j + n)  + \alpha + 2\right)  \prod_{j=1}^n (\beta_j + 1)}.$$
\end{prop}

Now we will give an explicit formula for the kernel $K_{\Omega_n^r}$ of $\Omega_n^r$.

\begin{theorem}\label{omega}
  If $w,\zeta \in \mathbb{D}^n$, then the Bergman kernel for $\Omega_n^r$ is
  \begin{align*}
  \pi^{n+1}  K_{\Omega_n^r}((0,w),(0,\zeta))=  \prod_{k=1}^{n} (1-\nu_k)^{-2} + \sum_{k=1}^{n} \frac{2(1+\nu_k)}{r (1-\nu_1)^2 \cdots (1-\nu_n)^2 (1-\nu_k)},
  \end{align*}
 where $\nu_k=w_k \overline{\zeta}_k$ for $k=1,\ldots,n$.
 Moreover if $\nu_1=\nu_2=\cdots=\nu_n$, then
 \begin{align*}
  \pi^{n+1}  K_{\Omega_n^r}((0,w),(0,\zeta))= \frac{(2n-r)\nu_1 + 2n+r}{r(1-\nu_1)^{2n+1}}.
  \end{align*}
 \end{theorem}

\begin{proof}
  As a consequence of Proposition \ref{pr3}, we have
  \begin{align*}
  \pi^{n+1}  K_{\Omega_n^r}((0,w),(0,\zeta))=  \sum_{\beta_1,\ldots,\beta_n \geq 0}^{\infty} \frac{  \Gamma\left(\frac{2}{r}(\sum_{j=1}^n \beta_j + n) + 2\right)\prod_{j=1}^n (\beta_j + 1) }{   \Gamma\left(\frac{2}{r}(\sum_{j=1}^n \beta_j + n) + 1\right) } \nu^{\beta},
  \end{align*}
  where $ \nu^{\beta}=\nu_1^{\beta_1} \cdot \cdots \cdot \nu_n^{\beta_n}$. Then, using the fact that $a\Gamma(a)=\Gamma(a+1)$,
     \begin{align*}
      \pi^{n+1}  K_{\Omega_n^r}((0,w),(0,\zeta))=  \sum_{\beta_1,\ldots,\beta_n \geq 0}^{\infty} \frac{2}{r} \left(\beta_1 + \cdots + \beta_n  +n + \frac{r}{2}\right)  \prod_{j=1}^n (\beta_j + 1) \nu^{\beta}.
     \end{align*}
   Hence, the kernel is
     \begin{align*}
     \frac{2}{r \pi^{n+1}} \sum_{k=0}^{n}   \sum_{\beta_1,\ldots,\beta_n \geq 0}^{\infty} \left(\beta_k + 1\right)  \prod_{j=1}^n (\beta_j + 1) \nu^{\beta},
     \end{align*} where $\beta_0=r/2-1$.
After some calculation using formulas
      \begin{align*}
        \sum_{m=0}^{\infty} (m+1)x^m=(1-x)^{-2} \quad \text{and} \quad \sum_{m=0}^{\infty} (m+1)^2x^m=\frac{1+x}{(1-x)^3},
     \end{align*}  we obtain desired formula.
     \end{proof}
 \subsection{Zeros of Bergman kernels on $\Omega_n^r$}
 In this section, we will prove that the Bergman kernel function of $\Omega_n^r$ for any natural number $n$ and positive real number $r$ is zero-free.
 \begin{prop}
   For any positive integer $n$, the domain $\Omega_n^r$ is a Lu Qi-Keng domain.
 \end{prop}

      \begin{proof} It is obvious that the Bergman kernel for $\Omega_1^r$ has no zeros.  The Bergman kernel function of $\Omega_n^r$ has zeros iff
      $$\frac{r}{2} + \sum_{k=1}^{n} \frac{1+\nu_k}{ 1-\nu_k}= n + \frac{r}{2} + \sum_{k=1}^{n} \frac{2\nu_k}{ 1-\nu_k} =0,$$
          for some $\nu_1,\ldots,\nu_n$ such that $|\nu_k|<1$, for each $k=1,\ldots,n$. Since $|\nu_k|<1$, then
          $$\Re \left( \frac{2\nu_k}{ 1-\nu_k} \right) > -1,$$
          where $\Re(\xi)$ is the real part of complex number $\xi$. Thus
                   $$\Re \left(  n + \frac{r}{2} + \sum_{k=1}^{n} \frac{2\nu_k}{ 1-\nu_k} \right) > 0.$$
      Hence, we see that $$\frac{r}{2} + \sum_{k=1}^{n} \frac{1+\nu_k}{ 1-\nu_k} \neq 0,$$
      if $(\nu_1,\ldots,\nu_n) \in \mathbb{D}^n$. The proof is therefore complete.
      \end{proof}


\small{

\bibliographystyle{amsplain}
}
\noindent Tomasz Beberok\\
Department of Applied Mathematics\\
University of Agriculture in Krakow\\
ul. Balicka 253c, 30-198 Krakow, Poland\\
email: tbeberok@ar.krakow.pl
\end{document}